\documentclass[12pt,oneside,american,english]{amsart}
\usepackage[T1]{fontenc}
\usepackage[latin9]{inputenc}
\usepackage{geometry}
\usepackage{babel}
\usepackage{prettyref}
\usepackage{units}
\usepackage{amsthm}
\usepackage{amstext}
\usepackage{amssymb}
\usepackage{graphicx}
\usepackage{setspace}
\usepackage[unicode=true,pdfusetitle,
 bookmarks=true,bookmarksnumbered=true,bookmarksopen=true,bookmarksopenlevel=1,
 breaklinks=true,pdfborder={0 0 0},backref=false,colorlinks=false]
 {hyperref}
\usepackage{breakurl}

\makeatletter

\providecommand{\tabularnewline}{\\}

\numberwithin{equation}{section}
\numberwithin{figure}{section}
  \theoremstyle{plain}
  \newtheorem*{thm*}{\protect\theoremname}
\theoremstyle{plain}
\newtheorem{thm}{\protect\theoremname}[section]
  \theoremstyle{definition}
  \newtheorem{defn}[thm]{\protect\definitionname}
  \theoremstyle{definition}
  \newtheorem*{example*}{\protect\examplename}
  \theoremstyle{plain}
  \newtheorem{lem}[thm]{\protect\lemmaname}
  \theoremstyle{plain}
  \newtheorem{cor}[thm]{\protect\corollaryname}
  \theoremstyle{plain}
  \newtheorem{prop}[thm]{\protect\propositionname}

\makeatother

  \addto\captionsamerican{\renewcommand{\corollaryname}{Corollary}}
  \addto\captionsamerican{\renewcommand{\definitionname}{Definition}}
  \addto\captionsamerican{\renewcommand{\examplename}{Example}}
  \addto\captionsamerican{\renewcommand{\lemmaname}{Lemma}}
  \addto\captionsamerican{\renewcommand{\propositionname}{Proposition}}
  \addto\captionsamerican{\renewcommand{\theoremname}{Theorem}}
  \addto\captionsenglish{\renewcommand{\corollaryname}{Corollary}}
  \addto\captionsenglish{\renewcommand{\definitionname}{Definition}}
  \addto\captionsenglish{\renewcommand{\examplename}{Example}}
  \addto\captionsenglish{\renewcommand{\lemmaname}{Lemma}}
  \addto\captionsenglish{\renewcommand{\propositionname}{Proposition}}
  \addto\captionsenglish{\renewcommand{\theoremname}{Theorem}}
  \providecommand{\corollaryname}{Corollary}
  \providecommand{\definitionname}{Definition}
  \providecommand{\examplename}{Example}
  \providecommand{\lemmaname}{Lemma}
  \providecommand{\propositionname}{Proposition}
  \providecommand{\theoremname}{Theorem}
\providecommand{\theoremname}{Theorem}

\begin{document}

\title{{\normalsize{}The Asymptotic Behavior of the Codimension Sequence
of Affine $G$-Graded Algebras }}

\author{Yuval Shpigelman}

\address{Department of Mathematics, Technion - Israel Institute of Technology,
Haifa 32000, Israel.}

\email{yuvalshp 'at' tx.technion.ac.il }

\keywords{graded algebras, polynomial identities, representation theory, Hilbert
series, codimension.}
\selectlanguage{american}%
\begin{abstract}
Let $W$ be an affine PI algebra over a field of characteristic zero
graded by a finite group $G.$ We show that there exist $\alpha_{1},\alpha_{2}\in\mathbb{R}$,
$\beta\in\frac{1}{2}\mathbb{Z}$, and $l\in\mathbb{N}$ such that
$\alpha_{1}n^{\beta}l^{n}\leq c_{n}^{G}(W)\leq\alpha_{2}n^{\beta}l^{n}$.
Furthermore, if $W$ has a unit then the asymptotic behavior of $c_{n}^{G}(W)$
is $\alpha n^{\beta}l^{n}$ where $\alpha\in\mathbb{R}$, $\beta\in\frac{1}{2}\mathbb{Z}$,
$l\in\mathbb{N}$. 
\end{abstract}
\maketitle
\selectlanguage{english}%

\section*{Introduction}

\selectlanguage{american}%
Throughout this article $F$ is an algebraically closed field of characteristic
zero and $W$ is an affine associative $F$- algebra graded by a finite
group $G$. We also assume that $W$ is a PI-algebra; i.e., it satisfies
an ordinary polynomial identity.\foreignlanguage{english}{ }In this
article we study $G$-graded polynomial identities of $W$, and in
particular the corresponding $G$-graded codimension sequence.\foreignlanguage{english}{
Let us briefly recall the basic setup. Let $X^{G}$ be a countable
set of variables $\left\{ x_{i,g}:g\in G;i\in\mathbb{N}\right\} $
and let $F\left\langle X^{G}\right\rangle $ be the free algebra on
the set $X^{G}$. Given a polynomial in $F\left\langle X^{G}\right\rangle $
we say that it is a $G$-graded identity of $W$ if it vanishes upon
any admissible evaluation on $W$. That is, a variable $x_{i,g}$
assumes values only from the corresponding homogeneous component $W_{g}$.
The set of all $G$-graded identities is an ideal in the free algebra
$F\left\langle X^{G}\right\rangle $ which we denote by $Id^{G}(W)$.
Moreover, $Id^{G}(W)$ is a $G$-graded $T$-ideal, namely, it is
closed under $G$-graded endomorphisms of $F\left\langle X^{G}\right\rangle $.
Let $R^{G}(W)$ denote the relatively free algebra $F\left\langle X^{G}\right\rangle /Id^{G}(W)$,
and $C_{n}^{G}(W)$ denote the space $P_{n}^{G}/P_{n}^{G}\cap Id^{G}(A)$
where} $P_{n}^{G}$ is the $|G|^{n}\cdot n!$ dimensional $F$-space
spanned by all permutations of the (multilinear) monomials $x_{1,g_{1}},x_{2,g_{2}}\cdot\cdot\cdot x_{n,g_{n}}$
where $g_{i}\in G$. We also define the $n$-th coefficient $c_{n}^{G}(W)$
of the codimension sequence of $W$ by $c_{n}^{G}(W)=dim_{F}\left(C_{n}^{G}(W)\right)$.\foreignlanguage{english}{
In \cite{Aljadeff & Belov2} Aljadeff and Belov showed that for every
affine $G$-graded algebra there exists a finite dimensional $G$-graded
algebra with the same ideal of graded identities, thus with the same
codimension sequence. We denote such algebra by $A_{W}$.}

In the last few years several papers have been written which generalize
results from the theory of ordinary polynomial identities to the $G$
-graded case (e.g. see \cite{Sviridova,Aljadeff & Belov1,Aljadeff & Belov2}).
One of these papers was by \foreignlanguage{english}{E.Aljadeff, A.Giambruno
and D.La Matina who showed, in different collaborations (see \cite{key-G+L,key-A+G,A+G+L}),
that, as in the nongraded case, $\lim_{n\to\infty}\sqrt[n]{c_{n}^{G}(W)}$
exists, and it is a nonnegative integer called the $G$ - graded exponent
of $W$, and denoted by $exp^{G}(W)$. In this paper we expand this
result and prove that, in addition to the ``exponential part'',
there is a ``polynomial part'' to the asymptotics of $c_{n}^{G}(W)$.
More precisely, we prove :}
\selectlanguage{english}%
\begin{thm*}[\textbf{\textit{A}}]
\label{Thm:A} Let $G$ be a finite group, and $W$ an affine $G$-
graded $F$-algebra where $F$ is a field of characteristic 0. Suppose
that $W$ satisfies an ordinary polynomial identity. Then there exist
$\alpha_{1},\alpha_{2}>0$, $\beta\in\frac{1}{2}\mathbb{Z}$, and
$l\in\mathbb{N}$ such that 
\[
\alpha_{1}n^{\beta}l^{n}\leq c_{n}^{G}(W)\leq\alpha_{2}n^{\beta}l^{n}.
\]

\end{thm*}
A conclusion of this theorem is that $\lim_{n\to\infty}log_{n}\left(\frac{c_{n}(W)}{exp(W)^{n}}\right)$
(the power of the ``polynomial part'') exists, and is an integer
or a half-integer.

Furthermore, if $W$ has a unit, we have found the structure of the
codimension sequence's asymptotics.
\begin{thm*}[\textbf{B}]
\label{thm:B} Let $G$ be a finite group, and $W$ an unitary affine
$G$- graded $F$-algebra where $F$ is a field of characteristic
0. Suppose that $W$ satisfies an ordinary polynomial identity. Then
there exist \foreignlanguage{american}{$\beta\in\frac{1}{2}\mathbb{Z}$,
$l\in\mathbb{N}$, and $\alpha\in\mathbb{R}$ }such that 

\[
c_{n}^{G}(W)\sim\alpha n^{\beta}l^{n}.
\]

\end{thm*}
Theorems \nameref{Thm:A} and \nameref{thm:B} generalize results
from the theory of ordinary polynomial identities. Indeed, in \cite{Berele & Regev}
Berele and Regev proved the nongraded version of theorems \nameref{Thm:A}
and \nameref{thm:B} for PI algebras satisfying a Capelli identity
(e.g affine algebras, see \cite{Gia & Zai book}). However, in Theorem
A (for ungraded algebras) Berele and Regev made an additional assumption,
namely, that the codimension sequence of the PI algebra is eventually
nondecreasing. Recently, in \cite{Gia&Za nondecreas}, Giambruno and
Zaicev showed that the codimension sequence of any PI algebra is eventually
nondecreasing and therefore the additional assumption mentioned above
can be removed. Here we prove a $G$-graded version of Giambruno and
Zaicev's result and so the assumption on the monotonicity of the $G$-greded
codimension sequence is also unnecessary.

We begin our article (\prettyref{sec:Preliminaries}) by mentioning
some notations from the representation theory of $GL_{n}(F)$ and
$S_{n}$. \prettyref{sec:Affine} is dealing with the case of affine
(not necessarily unitary) $G$ - graded algebras (Theorem A). And
\prettyref{sec:Unitary-Affine} is about the case of unitary affine
$G$-graded algebras (Theorem B). 

\selectlanguage{american}%

\section{\label{sec:Preliminaries} Preliminaries}

In this section we recall some notations and definitions from the
representation theory of $GL_{n}(F)$ and $S_{n}$.
\begin{defn}
A \textit{partition }is a finite sequence of integers $\lambda=(\lambda_{1},...,\lambda_{k})$
such that $\lambda_{1}\geq\cdots\geq\lambda_{k}>0$. The integer $k$
denoted by $h(\lambda)$, and called the \textit{height} of $\lambda$.
The set of all partitions is denoted by $\Lambda$, and the set of
all partitions of height less or equal to $k$ is denoted by $\Lambda^{k}$
.

We say that $\lambda$ is a\textit{ partition of $n\in\mathbb{N}$}
if $\sum_{i=1}^{k}\lambda_{i}=n$. In this case we write $\lambda\vdash n$
or $|\lambda|=n$. 
\end{defn}
It is known (e.g. see \cite{key-Bruce}) that the irreducible $S_{n}$
representations are indexed by partitions of $n$. Denote by $D_{\lambda}$
the $S_{n}$ irreducible representation indexed by $\lambda$, and
by $\chi_{\lambda}$ the corresponding character. By Maschke's theorem
every $S_{n}$ - representation $V$ is completely reducible, so we
can write $V=\bigoplus_{\lambda\vdash n}m_{\lambda}D_{\lambda}$,
and its character 
\[
\chi(V)=\sum_{\lambda\vdash n}m_{\lambda}\chi_{\lambda}.
\]

A $GL_{n}(F)$ - representation $Y=Sp_{F}\{y_{1},y_{2},...\}$ is
called a \textit{polynomial} if there is a set of polynomials $\left\{ f_{i,j}(z_{s,t})|i,j\in\mathbb{N}\right\} \subset F[z_{s,t}|1\leq s,t\leq n]$
such that for every $i$ only finite number of $f_{i,j}(z_{s,t})$'s
are nonzero, and the action of $GL_{n}(F)$ on $Y$ is given by:

\[
P\cdot y_{i}=\sum_{j}f_{i,j}(p_{s,t})y_{j}
\]
for every $P=(p_{s,t})\in GL_{n}(F)$. We say that $Y$ is \textit{$q-$homogeneous}
if all the nonzero $f_{i,j}$'s are homogeneous polynomials of (total)
degree $q$. For $\alpha=(\alpha_{1},...,\alpha_{n})\in\mathbb{N}^{n}$,
one defines the \textit{weight space} of $Y$ associated to $\alpha$
by 
\[
Y^{\alpha}=\left\{ y\in Y|\textrm{diag}(z_{1},...,z_{n})\cdot y=z_{1}^{\alpha_{1}}\cdots z_{n}^{\alpha_{n}}y\right\} .
\]
It is known that 
\[
Y=\bigoplus_{\alpha}Y^{\alpha}.
\]
The series $H_{Y}(t_{1},...,t_{n})=\sum_{\alpha}\left(\dim_{F}Y^{\alpha}\right)t_{1}^{\alpha_{1}}\cdots t_{n}^{\alpha_{n}}$
is called the \textit{Hilbert (or Poincare) series} of $Y$. The Hilbert
series is known to be symmetric in $t_{1},...,t_{n}$, so we recall
an important basis of the space of symmetric series, namely the Schur
functions. For convenience we use the following combinatorial definition:
(Note that although this definition is not the classical one , it
is equivalent to it e.g. see \cite{key-Bruce} ) 

Let $\lambda=(\lambda_{1},..,\lambda_{k})$ be a partition. The \textit{Young
diagram} associated with $\lambda$ is the finite subset of $\mathbb{Z}\times\mathbb{Z}$
defined as $\mathcal{D}_{\lambda}=\left\{ (i,j)\in\mathbb{Z}\times\mathbb{Z}|i=1,..,k\,,\, j=1,..,\lambda_{i}\right\} $.
We may regard $\mathcal{D}_{\lambda}$ as $k$ arrays of boxes where
the top one is of length $\lambda_{1}$, the second of length $\lambda_{2}$,
etc. For example 

\begin{center}
$\mathcal{D}_{(4,3,3,1)}=$ %
\begin{tabular}{|c|c|c|c}
\hline 
~ & ~ & ~ & \multicolumn{1}{c|}{~~}\tabularnewline
\hline 
\selectlanguage{english}%
\selectlanguage{american}%
 & \selectlanguage{english}%
\selectlanguage{american}%
 & \selectlanguage{english}%
\selectlanguage{american}%
 & \selectlanguage{english}%
\selectlanguage{american}%
\tabularnewline
\cline{1-3} 
\selectlanguage{english}%
\selectlanguage{american}%
 & \selectlanguage{english}%
\selectlanguage{american}%
 & \selectlanguage{english}%
\selectlanguage{american}%
 & \selectlanguage{english}%
\selectlanguage{american}%
\tabularnewline
\cline{1-3} 
\selectlanguage{english}%
\selectlanguage{american}%
 & \multicolumn{1}{c}{\selectlanguage{english}%
\selectlanguage{american}%
} & \multicolumn{1}{c}{\selectlanguage{english}%
\selectlanguage{american}%
} & \selectlanguage{english}%
\selectlanguage{american}%
\tabularnewline
\cline{1-1} 
\end{tabular}
\par\end{center}

\textit{A Schur function} $s_{\lambda}\in\mathbb{Z}[t_{1},...,t_{n}]^{S_{n}}$
is a polynomial such that the coefficient of $t_{1}^{a_{1}}\cdots t_{n}^{a_{n}}$
is equal to the number of ways to insert $a_{1}$ ones, $a_{2}$ twos,
... , and $a_{n}$ $n$'s in $\mathcal{D}_{\lambda}$ such that in
every row the numbers are non-decreasing, and in any column the numbers
are strictly increasing. Note that $s_{\lambda}$ is homogenous of
degree $|\lambda|$. 
\begin{example*}
~
\begin{enumerate}
\item If $\lambda$ is a partition of height one, i.e $\lambda=(\lambda_{1})$,
then the corresponding Schur function is 
\[
s_{(\lambda_{1})}(t_{1},...,t_{n})=\sum_{a_{1}+\cdots+a_{n}=\lambda_{1}}t_{1}^{a_{1}}\cdots t_{n}^{a_{n}}.
\]

\item If $\lambda=(2,1)$ and $n=2$ we have 
\[
s_{(2,1)}(t_{1},t_{2})=t_{1}^{2}t_{2}+t_{1}t_{2}^{2}
\]
since the only two ways to set ones and twos in $\mathcal{D}_{(2,1)}$
is %
\begin{tabular}{|c|c|}
\hline 
1 & 1\tabularnewline
\hline 
2 & \multicolumn{1}{c}{\selectlanguage{english}%
\selectlanguage{american}%
}\tabularnewline
\cline{1-1} 
\end{tabular} and %
\begin{tabular}{|c|c|}
\hline 
1 & 2\tabularnewline
\hline 
2 & \multicolumn{1}{c}{\selectlanguage{english}%
\selectlanguage{american}%
}\tabularnewline
\cline{1-1} 
\end{tabular}. 
\end{enumerate}
\end{example*}
The next lemma is easily implied from the definition of the Schur
functions. 
\begin{lem}
\label{lem:schur}~
\begin{enumerate}
\item $s_{\lambda}(t_{1},...,t_{n})=0$ if and only if $h(\lambda)>n$.
\item $s_{\lambda}(t_{1},....,t_{n},0,...,0)=s_{\lambda}(t_{1},...,t_{n})$.
\end{enumerate}
\end{lem}
As $\{s_{\lambda}(t_{1},...,t_{n})\}_{\lambda\in\Lambda^{n}}$ inform
a basis of the space of symmetric series in $n$ variables there exist
integers $\left\{ m_{\lambda}\right\} _{\lambda\in\Lambda^{n}}$ such
that 
\[
H_{Y}(t_{1},...,t_{n})=\sum_{\lambda\in\Lambda^{n}}m_{\lambda}s_{\lambda}(t_{1},...,t_{n})
\]
On the other hand, the group $S_{n}$ is embed in $GL_{n}(F)$ via
permutation matrices, hence $Y$ is also an $S_{n}$ representation.
Note that if $P$ is a permutation matrix there is $\sigma\in S_{n}$
such that $\mbox{diag\ensuremath{(z_{1},...,z_{n})\cdot P}}=P\cdot\mbox{diag\ensuremath{(z_{\sigma(1)},...,z_{\sigma(n)})}}$.
Hence for every $y$ in the weight space $Y^{(1^{n})}$ ($1^{n}$
means $1,...,1$ $n$ times) we have 
\[
\mbox{diag\ensuremath{(z_{1},...,z_{n})\cdot Py}}=P\cdot\mbox{diag\ensuremath{(z_{\sigma(1)},...,z_{\sigma(n)})y=z_{1}\cdots z_{n}Py}}
\]
so $Py\in Y^{(1^{n})}$, and we obtain that $Y^{(1^{n})}$ is a sub
$S_{n}$ representation of $Y$. Moreover, by the representation theory
of $GL_{n}(F)$ the $S_{n}$ character of $Y^{(1^{n})}$ is $\sum_{\lambda\vdash n}m_{\lambda}\chi_{\lambda}$
with the same coefficients $m_{\lambda}$ as in the Hilbert series
of $Y$. 

Next we recall some terminology and facts about $G$ - graded relatively
free algebras. Denote $X_{n}^{G}=\{x_{i,g}|g\in G\,,\,1\leq i\leq n\}\subset X^{G}$,
and consider the free algebra $F\left\langle X_{n}^{G}\right\rangle $
which is a subalgebra of $F\left\langle X^{G}\right\rangle $. We
equip $F\left\langle X_{n}^{G}\right\rangle $ with a $G$-grading
by setting the homogenous degree of a monomial $x_{i_{1},g_{i_{1}}}\cdots x_{i_{l},g_{i_{l}}}$
to be $g_{i_{1}}\cdots g_{i_{l}}\in G$. If we suppose that $\mathcal{I}$
is a $G$ - graded $T$-ideal of the free algebra $F\left\langle X_{n}^{G}\right\rangle $,
then it is easy to see that the grading on $F\left\langle X_{n}^{G}\right\rangle $
induces (naturally) a grading on $\nicefrac{F\left\langle X_{n}^{G}\right\rangle }{\mathcal{I}}$,
giving us the corresponding $T$-ideal of $G$-graded identities $Id^{G}\left(\nicefrac{F\left\langle X_{n}^{G}\right\rangle }{\mathcal{I}}\right)\subset F\left\langle X^{G}\right\rangle $.\\
We say that $\mathcal{I}$ is PI if $Id^{G}\left(\nicefrac{F\left\langle X_{n}^{G}\right\rangle }{\mathcal{I}}\right)\ne\{0\}$
i.e.$\nicefrac{F\left\langle X_{n}^{G}\right\rangle }{\mathcal{I}}$
is a PI algebra. Note that the $T$-ideal $F\left\langle X_{n}^{G}\right\rangle \cap Id^{G}\left(\nicefrac{F\left\langle X_{n}^{G}\right\rangle }{\mathcal{I}}\right)$
is exactly $\mathcal{I}$. 

For every $g\in G$ the group $GL_{n}(F)$ acts linearly on the $n$-dimensional
vector space $Sp_{F}\{x_{1,g},...,x_{n,g}\}$. Given these actions,
we define a polynomial $GL_{n}(F)$ representation structure on $F\left\langle X_{n}^{G}\right\rangle $
via $P(x_{i_{1},g_{1}}\cdots x_{i_{l},g_{l}})=P(x_{i_{1},g_{1}})\cdots P(x_{i_{l},g_{l}})$
for every $P\in GL_{n}(F)$. This structure induces (naturally) a
structure of $GL_{n}(F)$ representation on the relatively free algebra
$\nicefrac{F\left\langle X_{n}^{G}\right\rangle }{\mathcal{I}}$,
giving us the corresponding Hilbert series $H_{\nicefrac{F\left\langle X_{n}^{G}\right\rangle }{\mathcal{I}}}(t_{1},...,t_{n})$.
We recall from \cite{Aljadeff & Belov1} (lemma 2.10 and 2.11) the
following properties of that kind of Hilbert series:
\begin{lem}
\label{lem:hilbert-prop}Let $\mathcal{J}_{1},\mathcal{J}_{2}$ be
$G$ - graded $T$-ideals of $F\left\langle X_{n}^{G}\right\rangle $,
then
\begin{enumerate}
\item $H_{\nicefrac{F\left\langle X_{n}^{G}\right\rangle }{\mathcal{J}_{1}\cap\mathcal{J}_{2}}}=H_{\nicefrac{F\left\langle X_{n}^{G}\right\rangle }{\mathcal{J}_{1}}}+H_{\nicefrac{F\left\langle X_{n}^{G}\right\rangle }{\mathcal{J}_{2}}}-H_{\nicefrac{F\left\langle X_{n}^{G}\right\rangle }{\mathcal{J}_{1}+\mathcal{J}_{2}}}.$
\item If $\mathcal{J}_{1}\subset\mathcal{J}_{2}$ then 
\[
H_{\nicefrac{F\left\langle X_{n}^{G}\right\rangle }{\mathcal{J}_{1}}}=H_{\nicefrac{F\left\langle X_{n}^{G}\right\rangle }{\mathcal{J}_{2}}}+H_{\nicefrac{\mathcal{J}_{2}}{\mathcal{J}_{1}}}.
\]

\end{enumerate}
\end{lem}
In our case recall that $Id_{n}^{G}(W):=F\left\langle X_{n}^{G}\right\rangle \cap Id^{G}(W)$
is $G$-graded $T$-ideal which is PI for every $n\in\mathbb{N}$,
and let 
\[
\left\{ R_{n}^{G}(W)=\nicefrac{F\left\langle X_{n}^{G}\right\rangle }{Id_{n}^{G}(W)}\right\} _{n\in\mathbb{N}}
\]
be a collection of affine relatively free $G$-graded algebras of
$W$, and consider the corresponding Hilbert series $H_{R_{n}^{G}(W)}(t_{1},...,t_{n})=\sum_{\lambda}m_{\lambda}^{(n)}s_{\lambda}(t_{1},...,t_{n})$.
We can show that we need only one set of coefficients $\left\{ m_{\lambda}\right\} _{\lambda\in\Lambda}$
to describe $H_{R_{n}^{G}(W)}(t_{1},...,t_{n})$ for every $n$. Indeed,
it is easy to see that by definition diag$(z_{1},...,z_{n})\cdot x_{i_{1},g_{1}}\cdots x_{i_{l},g_{l}}=z_{i_{1}}\cdots z_{i_{l}}x_{i_{1},g_{1}}\cdots x_{i_{l},g_{l}}$,
hence the weight space $\left(R_{n}^{G}(W)\right)^{\alpha}$ is the
span of all the monomials in $R_{n}^{G}(W)$ with $\alpha_{i}$ appearances
of variables from the set $\{x_{i,g}|g\in G\}$ for every $i=1,...,n$.
Therefore, we can write $R_{r}^{G}(W)=\bigoplus\left(R_{m}^{G}(W)\right)^{\alpha}$
for every $r<m$, where the summing is over all the $\alpha\in\mathbb{N}^{m}$
such that $\alpha=(\alpha_{1},...,\alpha_{r},0,...,0)$. Thus, we
get the Hilbert series of $R_{r}^{G}(W)$ by putting $t_{r+1}=\cdots=t_{m}=0$
in the Hilbert series of $R_{m}^{G}(W)$, i.e. 
\[
H_{R_{r}^{G}(W)}(t_{1},...,t_{r})=\sum_{\lambda\in\Lambda^{r}}m_{\lambda}^{(r)}s_{\lambda}(t_{1},...,t_{r})=\sum_{\lambda\in\Lambda^{m}}m_{\lambda}^{(m)}s_{\lambda}(t_{1},...,t_{r},0,...0).
\]
Moreover, by \prettyref{lem:schur} we get
\[
\sum_{\lambda\in\Lambda^{r}}m_{\lambda}^{(r)}s_{\lambda}(t_{1},...,t_{r})=\sum_{\lambda\in\Lambda^{r}}m_{\lambda}^{(m)}s_{\lambda}(t_{1},...,t_{r})+\sum_{\lambda\in\Lambda^{m}\backslash\Lambda^{r}}m_{\lambda}^{(m)}\cdot0,
\]
hence for every $r<m$ the coefficients $m_{\lambda}^{(r)}$ and $m_{\lambda}^{(m)}$
are the same for every $\lambda\in\Lambda^{r}$. Making it possible
to denote $m_{\lambda}^{(n)}=m_{\lambda}$, and write for every $n\in\mathbb{N}$
\[
H_{R_{n}^{G}(W)}(t_{1},...,t_{n})=\sum_{\lambda\in\Lambda^{n}}m_{\lambda}s_{\lambda}(t_{1},...,t_{n}).
\]
Remember that $A_{W}$ is a finite dimensional $G$-graded algebra
such that $Id^{G}(A_{W})=Id^{G}(W)$, thus with the same $m_{\lambda}$.
By \cite{Gordienko} (Lemma 1 and Lemma 7) $m_{\lambda}=0$ if $\lambda\notin\Lambda^{k}$
where $k=dim_{F}\left(A_{W}\right)$, so all the nonzero coefficients
$m_{\lambda}$ appear in the series $H_{R_{k}^{G}(W)}(t_{1},...,t_{k})$.
Notice that for every $n$ the weight space $R_{n}^{G}(W)^{(1^{n})}$
is exactly $C_{n}^{G}(W)$, thus $c_{n}^{G}(W)=\sum_{\lambda\vdash n}m_{\lambda}d_{\lambda}$
where $d_{\lambda}=dim_{F}D_{\lambda}$. The conclusion is that the
$G$ - graded codimension sequence of $W$ is determined by a single
Hilbert series $H_{R_{k}^{G}(W)}(t_{1},...,t_{k})$. We denote this
series by $H_{W}(t_{1},...,t_{k})$.

As mention in the introduction, Berele and Regev proved versions of
Theorems \nameref{Thm:A} and \nameref{thm:B} for (nongraded) affine
PI algebras. Their proofs rely on a result about symmetric series
which we will also use. Before stating this result we need a definition.
\begin{defn}
A series $H\left(t_{1},...,t_{k}\right)$ is called \textit{rational
function} if we can write it as $\frac{f(t_{1},...,t_{k})}{g(t_{1},...,t_{k})}$
where $f,g$ are polynomials, moreover a rational function $\frac{f(t_{1},...,t_{k})}{g(t_{1},...,t_{k})}$
is called \textit{nice} if $g(t_{1},...,t_{k})=\prod_{\alpha\in\Delta}(1-t_{1}^{\alpha_{1}}\cdots t_{k}^{\alpha_{k}})$
where $\Delta$ is some subset of $\mathbb{N}^{k}$.\end{defn}
\begin{thm}[\cite{Berele & Regev}]
\label{thm:general} Let $H(t_{1},...,t_{k})=\sum_{\lambda\in\Lambda}a_{\lambda}s_{\lambda}(t_{1},...,t_{k})$
be a series, and denote $y_{n}^{H}=\sum_{\lambda\vdash n}a_{\lambda}d_{\lambda}$
where $d_{\lambda}$ is the dimension of the $\lambda$'s irreducible
representation of $S_{n}$, and $a_{\lambda}\in\mathbb{N}$. We assume
that 
\begin{enumerate}
\item There is an integer $r$ such that $a_{\lambda}=0$ for every partition
$\lambda$ satisfying $\lambda_{r+1}\geq M$ where $M$ is a constant
which depends only on $r$.
\item $H(t_{1},...,t_{k})$ is nice rational function.
\end{enumerate}
Then there is an integer $d$ such that for every $0\leq m\leq d-1$
there are constants $\beta_{m}\in\frac{1}{2}\mathbb{Z}$, and $\alpha_{m}\geq0$.
such that
\[
y_{n_{m,t}}^{H}\sim\alpha_{m}n_{m,t}^{\beta_{m}}l^{n_{m,t}}\,\,\,\,\,\,\,;\,\,\,\,\,\, t\to\infty
\]
where $n_{m,t}=m+td$, and $l$ is the minimal integer satisfying
condition 1. 
\end{thm}

\section{\label{sec:Affine}Affine G-graded Algebras}

In this section we prove Theorem A. Our first step is to show that
$H_{W}(t_{1},...,t_{k})=\sum_{\lambda\in\Lambda^{k}}m_{\lambda}s_{\lambda}$
satisfies the conditions of \prettyref{thm:general}, hence $y_{n}^{H_{W}}=c_{n}^{G}(W)$
satisfies the conclusion of the theorem. Recall that according to
the first condition there is an integer $r$ such that $m_{\lambda}=0$
for every partition $\lambda$ satisfying $\lambda_{r+1}\geq M$ where
$M$ is a constant which depends only on $r$. In \cite{Gordienko}
A.S Gordienko (see Lemma 1 and Lemma 7) showed that this condition
holds for finite dimensional $G$-graded PI $F$-algebras, however
we know (see the introduction) that there is a finite dimensional
$G$-graded algebra $A_{W}$ such that $H_{W}=H_{R_{k}^{G}(A_{W})}$,
so indeed $H_{W}$ satisfies the first condition of \prettyref{thm:general}.
Let us show that the second condition of \prettyref{thm:general}
holds, namely that $H_{W}$ is nice rational. In \cite{Aljadeff & Belov1}
Aljadeff and Belov proved that the Hilbert series of $\nicefrac{F\left\langle X_{k}^{G}\right\rangle }{\mathcal{I}}$
is rational for every $G$ - graded $T$ - ideal $\mathcal{I}$ of
$F\left\langle X_{k}^{G}\right\rangle $ which is PI. In fact, with
slight changes in their proof, we can prove that the Hilbert series
is also nice. 
\begin{thm}
Let $\mathcal{I}$ be a $G$ - graded $T$ - ideal of $F\left\langle X_{k}^{G}\right\rangle $
which is PI then the Hilbert series of $\nicefrac{F\left\langle X_{k}^{G}\right\rangle }{\mathcal{I}}$
is a nice rational function.\end{thm}
\begin{proof}
Assume by contradiction that $H_{\nicefrac{F\left\langle X_{k}^{G}\right\rangle }{\mathcal{I}}}(t_{1},...,t_{k})$
is not nice rational. Then the set of all PI $G$ - graded $T$-ideal
of $F\left\langle X_{k}^{G}\right\rangle $ whose Hilbert series is
not nice rational, is not empty. Furthermore, we may assume that there
exists a maximal element in this set, because otherwise (by Zorn's
lemma) there is an infi{}nite ascending sequence in this set. This
sequence does not stabilize, and hence its union is a nonfi{}nite
generated $T$-ideal. This of course contradicts the Specht property
of $G$- graded $T$-ideals (see \cite{Aljadeff & Belov2} section
12). 

So let $\mathcal{J}$ be a maximal PI $G$-graded $T$-ideal of $F\left\langle X_{k}^{G}\right\rangle $
such that $H_{\nicefrac{F\left\langle X_{k}^{G}\right\rangle }{\mathcal{J}}}$
is not nice rational. From \cite{Aljadeff & Belov2} (Theorem 11.2)
we know that 
\begin{equation}
Id^{G}\left(\nicefrac{F\left\langle X_{k}^{G}\right\rangle }{\mathcal{J}}\right)=Id^{G}(A_{1}\oplus\cdots\oplus A_{l})\label{eq:basic}
\end{equation}
where $A_{i}$ is a $G$-graded basic $F$-algebra for every $i$,
and $id^{G}(A_{i})\nsubseteq id^{G}(A_{j})$ for any $1\leq i,j\leq m$
with $i\ne j$. We refer the reader to \cite{Aljadeff & Belov2} for
more details - in particular, the definition of basic algebra. For
our purpose, we just have to use one property of basic algebras which
we will present later. Now, from \ref{eq:basic} we get that $\mathcal{J}=Id_{k}^{G}\left(\nicefrac{F\left\langle X_{k}^{G}\right\rangle }{\mathcal{J}}\right)=Id_{k}^{G}(A_{1})\cap Id_{k}^{G}(A_{2}\oplus\cdots\oplus A_{l})$.
Therefore, if we denote $\mathcal{J}_{1}=Id_{k}^{G}(A_{1})$ and $\mathcal{J}_{2}=Id_{k}^{G}(A_{2}\oplus\cdots\oplus A_{l})$
then we get from \prettyref{lem:hilbert-prop} 
\[
H_{\nicefrac{F\left\langle X_{k}^{G}\right\rangle }{\mathcal{J}}}=H_{\nicefrac{F\left\langle X_{k}^{G}\right\rangle }{\mathcal{J}_{1}}}+H_{\nicefrac{F\left\langle X_{k}^{G}\right\rangle }{\mathcal{J}_{2}}}-H_{\nicefrac{F\left\langle X_{k}^{G}\right\rangle }{\mathcal{J}_{1}+\mathcal{J}_{2}}}.
\]
Now, if $l>1$ then $\mathcal{J}\subsetneq\mathcal{J}_{1},\mathcal{J}_{2}$
from the maximality of $\mathcal{J}$ the series \\
$H_{\nicefrac{F\left\langle X_{k}^{G}\right\rangle }{\mathcal{J}_{1}}},\, H_{\nicefrac{F\left\langle X_{k}^{G}\right\rangle }{\mathcal{J}_{2}}}$
and $H_{\nicefrac{F\left\langle X_{k}^{G}\right\rangle }{\mathcal{J}_{1}+\mathcal{J}_{2}}}$
are all nice rational functions. Therefore,\\
$H_{\nicefrac{F\left\langle X_{k}^{G}\right\rangle }{\mathcal{J}}}$
is nice rational which is a contradiction. Hence, $l=1$, and
\[
Id^{G}\left(\nicefrac{F\left\langle X_{k}^{G}\right\rangle }{\mathcal{J}}\right)=Id^{G}(A)
\]
 where $A$ is a basic $F$-algebra.

We recall from \cite{Aljadeff & Belov1} (see section 2) that for
any basic algebra $A$ there is a $G$-graded $T$-ideal $\mathcal{K}$
strictly containing $\mathcal{J}=Id_{k}^{G}(A)$ such that $\nicefrac{\mathcal{K}}{\mathcal{J}}$
is a finite module over some affine commutative $F$-algebra, hence
by \cite{Stanly} (Chapter 1 Theorem 2.3) $H(\nicefrac{\mathcal{K}}{\mathcal{J}})$
is a nice rational function. Now, by the maximality of $\mathcal{J}$
the Hilbert series of $\nicefrac{F\left\langle X_{k}^{G}\right\rangle }{\mathcal{K}}$
is nice rational. Moreover, by \prettyref{lem:hilbert-prop} we get
\[
H_{\nicefrac{F\left\langle X_{k}^{G}\right\rangle }{\mathcal{J}}}=H_{\nicefrac{F\left\langle X_{k}^{G}\right\rangle }{\mathcal{K}}}+H_{\nicefrac{\mathcal{K}}{\mathcal{J}}}
\]
hence $H_{\nicefrac{F\left\langle X_{k}^{G}\right\rangle }{\mathcal{J}}}$
is nice rational which contradicts our assumption. 
\end{proof}
So indeed we can apply \prettyref{thm:general} on $H_{W}$, and we
arrive at the following corollary for affine $G$-graded algebra with
eventually non-decreasing codimension sequence (that is $c_{n}^{G}(W)\leq c_{n+1}^{G}(W)$
for every $n>N$ for some $N\in\mathbb{N}$). 
\begin{cor}
\label{cor:A}If the sequence $c_{n}^{G}(W)$ is eventually non-decreasing
then there exists $\alpha_{1},\alpha_{2}>0$, $\beta\in\frac{1}{2}\mathbb{Z}$,
and $l\in\mathbb{N}$ such that 
\[
\alpha_{1}n^{\beta}l^{n}\leq c_{n}^{G}(W)\leq\alpha_{2}n^{\beta}l^{n}.
\]
\end{cor}
\begin{proof}
We need to show that all the $\beta_{m}$'s form \prettyref{thm:general}
are equal. Suppose that there are $0\le m,l\leq d-1$ such that $\beta_{m}<\beta_{l}$,
and denote $s=l-m$ (mod $d$). Note that $s=m-l+qd$ thus 
\[
n_{m,t}+s=m+td+s=m+td+l-m+qd=l+(t+q)d=n_{l,t+q}
\]
 therefore 
\[
\lim_{t\to\infty}\frac{c_{n_{m,t}+s}^{G}(W)}{c_{n_{m,t}}^{G}(W)}=\lim_{t\to\infty}\frac{c_{n_{l,t+q}}^{G}(W)}{c_{n_{m,t}}^{G}(W)}=\lim_{t\to\infty}\frac{\alpha_{l}n_{l,t}{}^{\beta_{l}}l^{n_{l,t}}}{\alpha_{m}n_{m,t}^{\beta_{m}}l^{n_{m,t}}}=0
\]
 which is a contradiction since $c_{n}^{G}(W)$ is eventually non-decreasing. 
\end{proof}
The final stage of Theorem A's proof is removing the non-decreasing
assumption from \ref{cor:A}. 
\begin{prop}
\label{prop:nondec}Let $G$ be a finite group, and $W$ an affine
$G$- graded $F$-algebra where $F$ is a field of characteristic
0. Suppose that $W$ satisfying an ordinary polynomial identity. Then
the $G$- graded codimension sequence of $W$ is eventually non-decreasing. \end{prop}
\begin{proof}
Recall that there exists a finite dimensional algebra $A_{W}$ having
the same codimension sequence as $W$, so we may prove the proposition
for $A_{W}$. Let $A_{W}=S+J$ be the Wedderburn-Malcev decomposition
of $A_{W}$, and we may assume that the semi-simple part $S$, and
the Jacobson radical $J$ are $G$-graded. It is known (see \cite{Belov&Rowen book})
that $J$ is nilpotent, i.e. there is an integer $t$ such that $J^{t}=0$.
We claim that $c_{n}^{G}(A_{W})\leq c_{n+1}^{G}(A_{W})$ for every
$n>t$.

If $S=0$ then $A_{W}$ is nilponent of degree $t$, thus $c_{n}^{G}(A_{W})=0$
for every $n>t$. So assume that $S\ne0$, and let $1_{S}$ be its
unit element $1_{S}$. Because $1_{S}^{2}=1_{S}$, the linear operator
$T:A_{W}\to A_{W}$ ; $T(x)=x\cdot1_{S}$ has eigenvalues 1 or 0,
so we have the decomposition $A_{W}=A_{0}\oplus A_{1}$ where $A_{0},A_{1}$
are the eigenspaces of 0 and 1 respectively. Now, let $B=B_{0}\cup B_{1}$
be a basis of $A_{W}$ where $B_{0}\subset A_{0}$ and $B_{1}\subset A_{1}$.
We may assume that all the elements in the basis are $G$- homogenous.
Since $F$ is of characteristic $0$, it is enough to substitute the
variables in a polynomial $f\in F\left\langle X^{G}\right\rangle $
by element form $B$ in order to determine if $f$ is an identity
or not. 

Denote $G=\{g_{1},...,g_{s}=e\}$, and for any $\mathbf{h=}(h_{1},...,h_{n})\in G^{n}$
denote\\
 $P_{\mathbf{h}}=Sp_{F}\{x_{1,h_{1}},...,x_{n,h_{n}}\}$ and $C_{\mathbf{h}}=\nicefrac{P_{\mathbf{h}}}{P_{\mathbf{h}}\cap Id^{G}(A_{W})}$.
Note that $C_{n}^{G}(A_{W})=\bigoplus_{\mathbf{h}\in G^{n}}C_{\mathbf{h}}$,
and $C_{\mathbf{h}}\cong C_{\mathbf{k}}$ if the vector $\mathbf{k}$
is a permutation of the vector $\mathbf{h}$. Therefore, if we denote
$P_{(g_{1}^{n_{1}},...,g_{s}^{n_{s}})}=P_{n_{1},...,n_{s}}$ and $C_{(g_{1}^{n_{1}},...,g_{s}^{n_{s}})}=C_{n_{1},...,n_{s}}$
( $g_{i}^{n_{i}}$ means $g_{i},...,g_{i}$ $n_{i}$ times), we can
write 

\begin{equation}
c_{n}^{G}(A_{W})=\sum_{n_{1}+\cdots+n_{s}=n}{n \choose n_{1},...,n_{s}}c_{n_{1},...,n_{s}}\label{eq:binom}
\end{equation}
where ${n \choose n_{1},...,n_{s}}=\frac{n!}{n_{1}!\text{···}n_{s}!}$
is the generalized binomial coefficient and 
\[
c_{n_{1},...,n_{s}}=dim_{F}C_{n_{1},...,n_{s}}.
\]

Given $t<n=n_{1}+\cdots+n_{s}$, write $c_{n_{1},...,n_{s}}=k$ and
let\\
$f_{1}(x_{1,h_{1}},...,x_{n,h_{n}}),...,f_{k}(x_{1,h_{1}},...,x_{n,h_{n}})$
be a set of linear independent polynomials in $\nicefrac{P_{n_{1},...,n_{s}}}{P_{n_{1},...,n_{s}}\cap Id^{G}(W)}$
(note that $h_{1}=\cdots=h_{n_{1}}=g_{1}$, $h_{n_{1}+1}=\cdots=h_{n_{1}+n_{2}}=g_{2}$
and so on). For every $i=1,...k$ we define the following polynomials
in $P_{n_{1},...,n_{s}+1}$ (recall that $g_{s}=e$, so $P_{n_{1},...,n_{s}+1}=P_{(h_{1},...,h_{n},e)}$).
\[
p_{i}(x_{1,h_{1}},...,x_{n,h_{n}},x_{n+1,e})=\sum_{j=1}^{k}f_{i}(x_{1,h_{1}},...,x_{j,h_{j}}x_{n+1,e},...,x_{n,h_{n}}).
\]
We want to show that $p_{1},...,p_{k}\in\nicefrac{P_{n_{1},...,n_{s}+1}}{P_{n_{1},...,n_{s}+1}\cap Id^{G}(W)}$
are linearly independent, so let $\sum_{i=1}^{k}\alpha_{i}p_{i}$
be a nontrivial linear combination of the $p_{i}$'s. Since $f_{1},...,f_{k}$
are linear independent, there are $b_{j,h_{j}}\in B\cap\left(A_{W}\right)_{h_{j}}$
such that $\sum_{i=1}^{k}\alpha_{i}f_{i}(b_{1,h_{1}},...,b_{n,h_{n}})\ne0$.
Since $n$ is greater than the nilpotency index of $J$, one of the
$b_{i,h_{i}}$'s is necessarily in $S$, and $S\subset A_{1}$ so
$q=|B_{1}\cap\{b_{1,h_{1}},...,b_{n,h_{n}}\}|$ is a nonzero integer.
Let us substitute $b_{1,h_{1}},...,b_{n,h_{n}},1_{S}$ in $p_{i}$.
Note that $b_{j,h_{j}}\cdot1_{S}=b_{j,h_{j}}$ if $b_{j,h_{j}}\in B_{1}$,
and $b_{j,h_{j}}\cdot1_{S}=0$ if $b_{j,h_{j}}\in B_{0}$, therefore
\[
p_{i}(b_{1,h_{1}},...,b_{n,h_{n}},1_{S})=q\cdot f_{i}(b_{1,h_{1}},...,b_{n,h_{n}})
\]
and

\[
\sum_{i=1}^{k}\alpha_{i}p_{i}(b_{1,h_{1}},...,b_{n,h_{n}},1_{S})=q\cdot\sum_{i=1}^{k}\alpha_{i}f_{i}(b_{1,h_{1}},...,b_{n,h_{n}})\ne0.
\]
Thus the polynomial $\sum_{i=1}^{k}\alpha_{i}p_{i}(x_{1,h_{1}},...,x_{n,h_{n}},x_{n+1,e})$
is not an identity, so 
\[
c_{n_{1},...,n_{s}}=k\leq c_{n_{1},...,n_{s}+1}^{G}
\]
for every $n_{1},...,n_{s}$. Hence by \ref{eq:binom} 
\[
c_{n}^{G}(W)=\sum_{n_{1}+\cdots+n_{s}=n}{n \choose n_{1},...,n_{s}}c_{n_{1},...,n_{s}}
\]
\[
\leq\sum_{n_{1}+\cdots+n_{s}=n}{n+1 \choose n_{1},...,n_{s}+1}c_{n_{1},...,n_{s}+1}^{G}\leq c_{n+1}^{G}(W)
\]
where ${n \choose n_{1},...,n_{s}}\leq{n+1 \choose n_{1},...,n_{s}+1}$
is due to the combinatorial identity
\[
{n+1 \choose n_{1},...,n_{s}+1}={n \choose n_{1},...,n_{s}}+\sum_{i=1}^{s-1}{n \choose n_{1},...,n_{i}-1,...,n_{s}}.
\]

\end{proof}
From \prettyref{cor:A} and \prettyref{prop:nondec} we conclude:
\begin{thm*}[\textbf{\textit{A}}]
 Let $G$ be a finite group, and $W$ an affine $G$- graded $F$-algebra
where $F$ is a field of characteristic 0. Suppose $W$ satisfies
an ordinary polynomial identity. Then there exists $\alpha_{1},\alpha_{2}>0$,
$\beta\in\frac{1}{2}\mathbb{Z}$, and $l\in\mathbb{N}$ such that
\[
\alpha_{1}n^{\beta}l^{n}\leq c_{n}^{G}(W)\leq\alpha_{2}n^{\beta}l^{n}.
\]

\end{thm*}
A conclusion drawn from theorem A is that the codimension sequence's
asymptotics has a ``polynomial part'' of degree $\beta\in\frac{1}{2}\mathbb{Z}$.
More precisely,
\begin{cor}
The following limit exists and it is a half integer.

\[
\lim_{n\to\infty}log_{n}\left(\frac{c_{n}^{G}(W)}{exp^{G}(W)^{n}}\right)
\]

\end{cor}

\section{\label{sec:Unitary-Affine}Unitary Affine $G$-graded Algebras}

In this section we assume that $W$ has a unit. We show that, as in
the nongraded case, the asymptotics of the codimension sequence of
$W$ is a constant times polynomial part times exponential part (Theorem
B). 

In \cite{Drensky} V.Drensky showed that in the non-graded case (or
in another words for $G=\{e\}$) there is a subspace $B_{n}$ of $R_{n}^{\{e\}}(W)$
such that $R_{n}^{\{e\}}(W)=F[z_{1},...,z_{n}]\otimes B_{n}$ for
every $n\in\mathbb{N}$. We start this section by proving this result
for any finite group $G$. In order to keep the notations as light
as possible, whenever there is an element $f+I$ in some quotient
space $\nicefrac{A}{I}$ where $A\subset F\left\langle X_{n}^{G}\right\rangle $
we omit the $I$ and leave only the polynomial $f$ when the convention
is that all the operations on such polynomials are modulo $I$. So
let $B_{n}$ be the subspace of $R_{n}^{G}(W)$ spanned by all the
polynomials $x_{i_{1},h_{1}}\cdots x_{i_{r},h_{r}}[x_{i_{r+1},g_{1}},...,x_{i_{r+l},g_{l}}]\cdots[...,x_{i_{s},g_{s}}]$
where $h_{j}\ne e$ for $i=1,...,r$. Note that $B_{n}$ is clearly
a sub polynomial $GL_{n}(F)$ representation of $R_{n}^{G}(W)$, so
the Hilbert series $H_{B_{n}}(t_{1},...,t_{n})=\sum_{\lambda\in\Lambda^{n}}a_{\lambda}^{(n)}s_{\lambda}(t_{1},...,t_{n})$
is well defined. Moreover, we can write $B_{r}=\bigoplus B_{m}^{\alpha}$
for every $r<m$, where the sum is over all the $\alpha\in\mathbb{N}^{m}$
such that $\alpha=(\alpha_{1},...,\alpha_{r},0,...,0)$. Thus, we
can get the Hilbert series of $B_{r}$ by putting $t_{r+1}=\cdots=t_{m}=0$
in the Hilbert series of $B_{m}$, i.e. 
\[
H_{B_{r}}(t_{1},...,t_{r})=\sum_{\lambda\in\Lambda^{r}}a_{\lambda}^{(r)}s_{\lambda}(t_{1},...,t_{r})=\sum_{\lambda\in\Lambda^{m}}a_{\lambda}^{(m)}s_{\lambda}(t_{1},...,t_{r},0,...0).
\]
Moreover, from \prettyref{lem:schur}
\[
\sum_{\lambda\in\Lambda^{r}}a_{\lambda}^{(r)}s_{\lambda}(t_{1},...,t_{r})=\sum_{\lambda\in\Lambda^{r}}a_{\lambda}^{(m)}s_{\lambda}(t_{1},...,t_{r})+\sum_{\lambda\in\Lambda^{m}\backslash\Lambda^{r}}a_{\lambda}^{(m)}\cdot0.
\]
Hence, for every $r<m$ the coefficients $a_{\lambda}^{(r)}$ and
$a_{\lambda}^{(m)}$ are the same for every $\lambda\in\Lambda^{r}$.
Making it possible to denote $a_{\lambda}^{(n)}=a_{\lambda}$, and
write for every $n\in\mathbb{N}$ 
\[
H_{B_{n}}=\sum_{\lambda\in\Lambda^{n}}a_{\lambda}s_{\lambda}(t_{1},...,t_{n}).
\]
We wish to show that, as in $R_{n}^{G}(W)$, a single Hilbert series
determines all the coefficients $\left\{ a_{\lambda}\right\} _{\lambda\in\Lambda}$,
and that \prettyref{thm:general} can be applied to this series. Fix
an integer $t$, and recall that for every polynomial $GL_{n}(F)$
representation $Y$ the $t$ homogenous sub representation of $Y$
is $Y^{(t)}=\bigoplus_{|\alpha|=t}Y^{\alpha}$. Let us define for
every $s=0,...,t$ the spaces $Q_{s}=R_{n,e}^{(t-s)}B_{n}^{(s)}\subset\left(R_{n}^{G}(W)\right)^{(t)}$
where $R_{n,e}=\nicefrac{F\left\langle x_{1,e},...,x_{n,e}\right\rangle }{Id^{G}(W)\cap F\left\langle x_{1,e},...,x_{n,e}\right\rangle }$,
and the spaces $M_{p}=\sum_{s=p}^{t}Q_{s}$ for $p=0,..t$ (and $M_{t+1}=\{0\}$).
It is easy to see that $Q_{s}$ and $M_{p}$ are both polynomial $GL_{n}(F)$
representations for every $s$ and $p$. 
\begin{lem}
\label{lem:Decomposition}~
\begin{enumerate}
\item $\left(R_{n}^{G}(W)\right)^{(t)}=M_{0}$
\item $\left(R_{n}^{G}(W)\right)^{(t)}\cong\bigoplus_{p=0}^{t}\nicefrac{M_{p}}{M_{p+1}}$
as $GL_{n}(F)$ modules. 
\end{enumerate}
\end{lem}
\begin{proof}
It is clear that $M_{0}\subset\left(R_{n}^{G}(W)\right)^{(t)}$. For
the other inclusion, note that the free algebra $F\left\langle X_{n}^{G}\right\rangle $
is a universal enveloping of the free Lie algebra on $X_{n}^{G}$.
Hence, according to Poincare-Birkhoff-Witt theorem, if $l_{1},l_{2},...$
is a totally ordered basis of the free Lie algebra, then $F\left\langle X_{n}^{G}\right\rangle $
has a basis $\left\{ l_{1}^{\beta_{1}}\cdots l_{r}^{\beta_{r}}|\,\,\beta_{i}\geq0\right\} $.
In our case we choose $l_{i}=x_{i,e}$ for $i=1,...,n$, and the other
basis elements are arranged in an order such that $deg(l_{j})\leq deg(l_{j+1})$
for every $j>n$. We conclude that there is a basis $K$ of $F\left\langle X_{n}^{G}\right\rangle $
consisting of polynomials of the form $x_{1,e}^{\alpha_{1}}\cdots x_{n,e}^{\alpha_{n}}x_{i_{1},h_{1}}\cdots x_{i_{r},h_{r}}[x_{i_{r+1},g_{1}},...]\cdots[...,x_{i_{s},g_{s}}]$
which is in $R_{n,e}^{(|\alpha|)}B_{n}^{(s)}$ since $h_{i}\ne e$
for every $i=1,...,r$. Moreover, the quotient map is a $G$-graded
homomorphism from $F\left\langle X_{n}^{G}\right\rangle ^{(t)}$ onto
$\left(R_{n}^{G}(W)\right)^{(t)}$, thus $\left(R_{n}^{G}(W)\right)^{(t)}$
is spanned by the $t$-homogenous elements of $K$ (modulo the identities). 

For the second part recall that every $GL_{n}(F)$ representation
is completely reducible, i.e. $\left(R_{n}^{G}(W)\right)^{(t)}=\bigoplus_{i\in I}V_{i}$
where $I$ is a finite set of indexes and $V_{i}$ is irreducible
$GL_{n}(F)$ representation for every $i\in I$. Moreover, $Q_{s}=\bigoplus_{i\in I_{s}}V_{i}$
where $I_{s}$ is a subset of $I$, and $M_{p}=\bigoplus_{i\in\cup_{s=p}^{t}I_{s}}V_{i}$.
Therefore, $\nicefrac{M_{p}}{M_{p+1}}\cong\bigoplus_{i\in J_{p}}V_{i}$
where $J_{p}=\bigcup_{s=p}^{t}I_{s}-\bigcup_{s=p+1}^{t}I_{s}$. So
to prove the second part we need to show that $I$ is a disjoint union
of $J_{p}$ where $p=0,...,t$. The inclusion $\bigcup_{p=0}^{t}J_{p}\subset I$
is trivial. For the other direction, recall that by part one $I=\bigcup_{s=0}^{t}I_{s}$,
so for every $i\in I$ there is a minimal $p$ such that $i\in I_{p}$.
Hence $i$ belongs to $\bigcup_{s=p}^{t}I_{s}$ and not to $\bigcup_{s=p+1}^{t}I_{s}$.
By definition, $i\in J_{p}$, and we conclude that $I=\bigcup_{p=0}^{t}J_{p}$.
To show that the $J_{p}$'s are disjoint suppose that there is an
index $i$ in $J_{p}\cap J_{p'}$ for $p>p'$. The set $J_{p}$ is
a subset of $I_{p}$, and $p\geq p'+1$, thus $i\in\bigcup_{s=p'+1}^{t}I_{s}$.
Hence, by definition, $i\notin J_{p'}$, which is a contradiction.
\end{proof}
The next step is to find a basis of $\nicefrac{M_{p}}{M_{p+1}}$ for
every $p\leq t$. In order to achieve this, we need some preparations:
\begin{lem}
\label{lem:comm}Let $\{y_{1},...,y_{k}\}$ be a set of non commutative
variables where $k\geq2$, then we have the following identity 
\begin{equation}
y_{1}\cdots y_{k}=y_{2}y_{1}y_{3}\cdots y_{k}+\sum_{i=0}^{k-2}\sum y_{j_{1}}\cdots y_{j_{i}}[y_{l_{1}},...,y_{l_{k-i}}]\label{id}
\end{equation}
where the inner sum is over some subset of 
\[
\left\{ (j_{1},...,j_{i},l_{1},...,l_{k-i})\in\mathbb{N}^{k}\,|\,\{j_{1},...,j_{i},l_{1},...,l_{k-i}\}=\{1,...,k\}\right\} .
\]
\end{lem}
\begin{proof}
By induction on $k$. The base of the induction, $k=2$, is true by
the definition of the commutator $y_{1}y_{2}=y_{2}y_{1}+[y_{1},y_{2}]$.
For the induction step let us assume that the claim is true for $k$.
So, by multiplying the identity \ref{id} from the right by $y_{k+1}$
we get: 
\[
y_{1}\cdots y_{k+1}=y_{2}y_{1}y_{3}\cdots y_{k+1}+\sum_{i=0}^{k-2}\sum y_{j_{1}}\cdots y_{j_{i}}[y_{l_{1}},...,y_{l_{k-i}}]y_{k+1}.
\]
We complete the proof by using the commutator definition $zw=wz+[z,w]$
with \\
$z=[y_{1},...,y_{l_{k-i}}]$ and $w=y_{k+1}$ for every element in
the sum. \end{proof}
\begin{defn}
We say that $f\in F\left\langle X_{n}^{G}\right\rangle $ is \emph{$g$
multi-homogenous of degree} $\gamma=(\gamma_{1},...,\gamma_{n})$,
if substituting the variables $x_{i,g}$ by $t_{i}x_{i,g}$ in $f$
gives 
\[
\left(\prod_{i=1}^{n}t_{i}^{\gamma_{i}}\right)f.
\]
In other words for every $i=1,..,n$ the variable $x_{i,g}$ appears
$\gamma_{i}$ times in every monomial of $f$. We denote the $g$
multi-homogenous degree of polynomial $f$ by $deg_{g}(f)$. 
\end{defn}
It is known that for any $g\in G$ if $f=\sum_{\gamma}f_{\gamma}$
is a $G$-graded identity, where $f_{\gamma}$ is $g$ multi-homogenous
of degree $\gamma$, then $f_{\gamma}$ is also $G$-graded identity
for every $\gamma$. Therefore, the decomposition $B_{n}^{(p)}=\bigoplus_{\gamma}\left(B_{n}^{(p)}\right)_{\gamma}$
where $\left(B_{n}^{(p)}\right)_{\gamma}$ is the subspace of $B_{n}^{(p)}$
consisting of $g$ multi-homogenous polynomials of degree $\gamma$.
We may therefore assume that for every $g\in G$ there is a basis$\left\{ f_{j}^{(p)}(x_{l,h})\,\,|j=1,...,dim_{F}B_{n}^{(p)}\right\} $
of $B_{n}^{(p)}$ consisting of $g$ multi-homogenous polynomials.
Now we are ready to present the basis of $\nicefrac{M_{p}}{M_{p+1}}$.
\begin{prop}
\label{prop:Drensky} For every $p\leq t$ let $\left\{ f_{j}^{(p)}(x_{l,h})\,\,|j=1,...,dim_{F}B_{n}^{(p)}\right\} $
be a basis of $B_{n}^{(p)}$ consisting of $e$ multi-homogenous polynomials.
Then:
\begin{enumerate}
\item For every $1\leq j\leq dim_{F}B_{n}^{(p)}$ and $\sigma\in S_{t-p}$
\[
x_{i_{1},e}\cdots x_{i_{t-p},e}f_{j}^{(p)}(x_{l,h})\equiv x_{i_{\sigma(1)},e}\cdots x_{i_{\sigma(t-p)},e}f_{j}^{(p)}(x_{l,h})\mbox{ (mod \ensuremath{M_{p+1}})}
\]

\item The set 
\[
U_{p}=\left\{ x_{1,e}^{\alpha_{1}}\cdots x_{n,e}^{\alpha_{n}}f_{j}^{(p)}(x_{l,h})\,\,|\,|\alpha|=t-p,\, j=1,...,dim_{F}B_{n}^{(p)}\right\} 
\]
 is a basis of \textup{$\nicefrac{M_{p}}{M_{p+1}}$}.
\end{enumerate}
\end{prop}
\begin{proof}
To prove the first part it is enough to show that for every $1\leq r\leq t-p-1$
\[
x_{i_{1},e}\cdots x_{i_{t-p},e}f_{j}^{(p)}(x_{l,h})\equiv x_{i_{1},e}\cdots x_{r+1,e}x_{r,e}\cdots x_{i_{t-p},e}f_{j}^{(p)}(x_{l,h})\mbox{ (mod \ensuremath{M_{p+1}})}.
\]
Recall that $f_{j}^{(p)}$ is a linear combination of polynomials
of the form\\
$x_{i_{1},h_{1}}\cdots x_{i_{q},h_{q}}[x_{l_{1},g_{1}},...,][...,x_{l_{p-q},g_{p-q}}]$
where $h_{i}\ne e$. Hence, by using \prettyref{lem:comm} on $x_{r,e}\cdots x_{i_{t-p},e}x_{i_{1},h_{1}}\cdots x_{i_{q},h_{q}}$
for every such monomial, we obtain 
\[
x_{i_{1},e}\cdots x_{i_{t-p},e}f_{j}^{(p)}(x_{l,h})=x_{i_{1},e}\cdots x_{i_{r+1},e}x_{i_{r},e}\cdots x_{i_{t-p},e}f_{j}^{(p)}(x_{l,h})+P(x_{l,h})
\]
where $P(x_{l,h})$ is a linear combination of polynomials of the
form\\
$x_{j_{1}e},...,x_{j_{t-q-k},e}x_{i_{1},h_{1}}\cdots x_{i_{q},h_{q}}[x_{l_{1},g_{1}},...,][...,x_{l_{k},g_{k}}]$
with $k+q>p$. Thus $P(x_{l,h})\in M_{p+1}$, and the first part is
proven. 

Obviously $U_{p}$ spans $\nicefrac{M_{p}}{M_{p+1}}$. It remains
to be show that $U_{p}$ is linearly independent, so let 
\[
T=\sum_{j,\alpha}\beta_{j,\alpha}x_{1,e}^{\alpha_{1}}\cdots x_{n,e}^{\alpha_{n}}f_{j}^{(p)}(x_{l,h})
\]
be a nontrivial linear combination of the elements in $U_{p}$. Assume
by contradiction that $T$ is an identity of $W$. Therefore $T_{\gamma}$,
the $e$ multi-homogenous component of degree $\gamma$ of $T$, is
also an identity for every $\gamma\in\mathbb{N}^{n}$. We choose $\gamma$
such that $T_{\gamma}$ is a nontrivial linear combination, and write
\[
T_{\gamma}=\sum_{j}\sum_{\alpha=\gamma-deg_{e}\left(f_{j}^{(p)}\right)}\beta_{j,\alpha}x_{1,e}^{\alpha_{1}}\cdots x_{n,e}^{\alpha_{n}}f_{j}^{(p)}
\]
where the first sum is over every $1\leq j\leq dim_{F}B_{n}^{(p)}$
such that $deg_{e}\left(f_{j}^{(p)}\right)_{i}\leq\gamma_{i}$ for
every $i=1,...,n$ (note that every index $j$ is appearing at most
once in the sum). 

Now, let substitute $x_{i,e}$ by $x_{i,e}+1$ in $T_{\gamma}$ and
denote the obtained identity by $S$. Notice that since $[y+1,z]=[y,z]$
the polynomials $f_{j}^{(p)}$ are stay the same under this substitution
(in the part which is not commutator products in $f_{j}^{(p)}$ there
are no $e$-homogenous variables). Therefore, the $p$ - homogenous
component of $S$ is 
\[
\sum_{j}\beta_{j,\alpha}f_{j}^{(p)}
\]
which is also an identity of $W$. But $f_{j}^{(p)}$ are linear independent
(modulo the identities) and not all the $\beta_{j,\alpha}$ are zero,
and this is a contradiction .\end{proof}
\begin{cor}
\label{cor:OfDrensky} .
\begin{enumerate}
\item $R_{n}^{G}(W)\cong F[z_{1},...,z_{n}]\otimes B_{n}$ as $GL_{n}(F)$
representation.
\item $c_{n}^{G}(W)=\sum_{s=0}^{n}{n \choose s}\delta_{s}$ where $\delta_{s}=dim_{F}B_{s}^{(1^{s})}$
.
\end{enumerate}
\end{cor}
\begin{proof}
By \prettyref{prop:Drensky} the mapping from $\nicefrac{M_{p}}{M_{p+1}}$
to $F[z_{1},...,z_{n}]^{(t-p)}\otimes B_{n}^{(p)}$ taking $\left(\prod_{i=1}^{n}x_{i,e}^{\alpha_{i}}\right)f_{j}^{(p)}$
to $\left(\prod_{i=1}^{n}z_{i}^{\alpha_{i}}\right)\otimes f_{j}^{(p)}$
is a $GL_{n}(F)$ representation isomorphism, thus by \prettyref{lem:Decomposition}
$R_{n}^{G}(W)^{(t)}\cong\left(F[z_{1},...,z_{n}]\otimes B_{n}\right)^{(t)}$
for every $t$. 

For the proof of the second part, note that 
\[
C_{n}^{G}(W)\cong\left(F[z_{1},...,z_{n}]\otimes B_{n}\right)^{(1^{n})}=\bigoplus_{\beta\in\{0,1\}^{n}}F[z_{1},...,z_{n}]^{(1^{n})-\beta}\otimes B_{n}^{\beta}
\]
The dimension of $F[z_{1},...,z_{n}]^{(1^{n})-\beta}$ is $1$ for
every $\beta\in\{0,1\}^{n}$. Furthermore, it is easy to see that
$B_{n}^{\beta}\cong B_{|\beta|}^{(1^{|\beta|})}$ for every $\beta\in\{0,1\}^{n}$,
and there are ${n \choose s}$ vectors $\beta\in\{0,1\}^{n}$ such
that $|\beta|=s$. Thus $dim_{F}C_{n}^{G}(W)=c_{n}^{G}(W)=\sum_{s=0}^{n}{n \choose s}\delta_{s}$.
\end{proof}
With this corollary we can write 
\[
H_{R_{n}^{G}(W)}=H_{F[z_{1},...,z_{n}]}\cdot H_{B_{n}}.
\]
Recall that $H_{R_{n}^{G}(W)}(t_{1},...,t_{n})=\sum_{\lambda}m_{\lambda}s_{\lambda}(t_{1},...,t_{n})$
and $H_{B_{n}}(t_{1},...,t_{n})=\sum_{\lambda}a_{\lambda}s_{\lambda}(t_{1},...,t_{n})$.
Moreover, the Hilbert series of $F[z_{1},...,z_{n}]$ is known to
be 
\begin{equation}
\sum_{\alpha\in\mathbb{N}^{n}}t_{1}^{\alpha_{1}}\cdots t_{n}^{\alpha_{n}}=\prod_{i=1}^{n}(1-t_{i})^{-1}=\sum_{k\in\mathbb{N}}s_{(k)}(t_{1},...,t_{n}),\label{eq:HilbertRat}
\end{equation}
thus we have the identity

\begin{equation}
\sum_{\lambda}m_{\lambda}s_{\lambda}=\sum_{\mu\in\Lambda}\,\sum_{k\in\mathbb{N}}a_{\mu}s_{\mu}s_{(k)}\label{eq:RnHilbert}
\end{equation}
This identity gives us the following arithmetic connection between
$m_{\lambda}$ and $a_{\mu}$.
\begin{prop}
\label{prop:a-vr-m}For every partition $\lambda$ 
\[
m_{\lambda}=\sum_{\mu\in L_{\lambda}}a_{\mu}
\]
where $L_{\lambda}$ is the set of all partition $\mu$ such that
$\lambda_{i+1}\leq\mu_{i}\leq\lambda_{i}$ for every $i=1,...,h(\mu)$. \end{prop}
\begin{proof}
By Pieri's formula (see \cite{Macdonald}) $s_{\mu}s_{(k)}=\sum s_{\nu}$
where the sum is over all partitions $\nu$ obtained from $\mu$ by
adding $k$ boxes to its Young's diagram, no two of them in the same
column. We have to show that $s_{\lambda}$ appearing in the sum on
right hand side of $s_{\mu}s_{(|\lambda|-|\mu|)}=\sum s_{\nu}$ if
and only if $\lambda_{i+1}\leq\mu_{i}\leq\lambda_{i}$ for every $i$. 

Suppose that $s_{\lambda}$ appearing in the sum on the right hand
side of $s_{\mu}s_{(|\lambda|-|\mu|)}=\sum s_{\nu}$. It is obvious
that $\mu_{i}\leq\lambda_{i}$ for every $i$, since we are adding
boxes to $\mu$ in order to get $\lambda$. Let us assume that there
is an index $i$ such that $\mu_{i}<\lambda_{i+1}\leq\lambda_{i}$.
In the$\mu_{i}+1$ column, then, we necessarily added two boxes which
is a contradiction. (see Figure \ref{fig:tablux} )

\begin{figure}[h]
\centering{}\caption{\selectlanguage{english}%
\label{fig:tablux}\protect\includegraphics[scale=0.5]{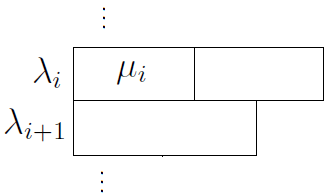}\selectlanguage{american}%
}
\end{figure}

For the other direction, let $\lambda$ and $\mu$ be partitions such
that $\lambda_{i+1}\leq\mu_{i}\leq\lambda_{i}$ for every $i$. Suppose
that in the $j$'s column we added two boxes in the $k$ 's and $i$'s
row where $k>i$. This implies that $\mu_{i}<j\leq\lambda_{k}\leq\lambda_{i+1}$
and this is a contradiction.
\end{proof}
The conclusion of the above proposition is that the coefficients $a_{\lambda}$
inherit two properties from the coefficients $m_{\lambda}$. Here
is the first one. 
\begin{cor}
There exists an integer $q$ such that $a_{\lambda}=0$ if $\lambda\notin\Lambda^{q}$.\end{cor}
\begin{proof}
Recall from \prettyref{sec:Preliminaries} that $m_{\lambda}=0$ if
$\lambda\notin\Lambda^{dim_{F}(A_{W})}$ where $A_{W}$ is the finite
dimensional algebra with the same $G$-graded identities as $W$. 

Let $\mu$ be a partition such that $h(\mu)>dim_{F}(A_{W})-1$ (i.e.
$\mu\notin\Lambda^{dim_{F}A_{W}-1})$, and define $\lambda=(\mu_{1},...,\mu_{h(\mu)},1).$
Since $h(\lambda)=h(\mu)+1>dim_{F}(A_{W})$ we know that $m_{\lambda}=\sum_{\nu\in L_{\lambda}}a_{\nu}=0$,
so $a_{\nu}=0$ for any $\nu\in L_{\lambda}$ . Moreover, it is easy
to see that $\lambda_{i+1}\leq\mu_{i}\leq\lambda_{i}$ for every $i$
thus $\mu\in L_{\lambda}$.
\end{proof}
By this corollary, and since $H_{B_{n}}(t_{1},...,t_{n})=\sum_{\lambda\in\Lambda^{n}}a_{\lambda}s_{\lambda}(t_{1},...,t_{n})$
for every $n$, we conclude that there is a single Hilbert series
(here it is $H_{B_{q}}(t_{1},...,t_{q})$) which determines all the
coefficients $a_{\lambda}$. The second corollary of \prettyref{prop:a-vr-m}
is that, as in $R_{n}^{G}(W)$ (see \prettyref{sec:Preliminaries}),
the series $H_{B_{q}}(t_{1},..,t_{q})$ satisfies the first assumption
of \prettyref{thm:general}. 
\begin{cor}
Let $H_{B_{q}}(t_{1},...,t_{q})=\sum_{\lambda}a_{\lambda}s_{\lambda}(t_{1},...,t_{q})$.
Then there is an integer $r$ such that $a_{\lambda}=0$ for every
partition $\lambda$ satisfying $\lambda_{r+1}\geq M$ where $M$
is a constant depend only on $r$. 

Moreover the integer $r=exp^{G}(W)-1$ is the minimal satisfying this
condition.\end{cor}
\begin{proof}
From \cite{Gordienko} there exists an integer $M$ such that $m_{\lambda}=0$
for every partition $\lambda$ satisfies $\lambda_{l+1}\geq M$ where
$l=exp^{G}(W).$ Suppose that $\mu$ is a partition such that $\mu_{l}\geq M$.
We define a partition $\lambda=(\mu_{1},...,\mu_{l},\mu_{l},\mu_{l+2},...,\mu_{h(\mu)})$.
Obviously, $\lambda_{l+1}=\mu_{l}\geq M$, so $m_{\lambda}=\sum_{\nu\in L_{\lambda}}a_{\nu}=0$,
so $a_{\nu}=0$ for any $\nu\in L_{\lambda}$. Moreover, one can check
that $\lambda_{i+1}\leq\mu_{i}\leq\lambda_{i}$ for every $i$ so
$\mu\in L_{\lambda}$. 

To show the minimality of $l-1$, we have to show that for every integer
$p$ there is a partition $\mu$ with $\mu_{l-1}\geq p$ such that
$a_{\mu}\ne0$. We know from \cite{Gordienko} that for every integer
$p$ there is a partition $\lambda$ with $\lambda_{l}\geq p$ such
that $m_{\lambda}=\sum_{\nu\in L_{\lambda}}a_{\nu}\ne0$. Hence, there
is $\mu\in L_{\lambda}$, which satisfies $\mu_{l-1}\geq\lambda_{l}\geq p$,
such that $a_{\mu}\ne0$. 
\end{proof}
The Hilbert series $H_{B_{q}}$ is also nice rational because $H_{R_{q}^{G}(W)}$
is a nice rational function, and since \ref{eq:HilbertRat} imply
that 
\[
H_{B_{q}}(t_{1},...,t_{q})=\prod_{i=1}^{q}(1-t_{i})H_{R_{q}^{G}(W)}(t_{1},...,t_{q}).
\]
We get that the conclusion of \prettyref{thm:general} holds for the
series $\{\delta_{s}\}_{s\in\mathbb{N}}$, and together with \prettyref{cor:OfDrensky}
\[
c_{n}^{G}(W)\sim\sum_{m=0}^{d-1}\,\,\,\sum_{\begin{array}{c}
0\leq s\leq n\\
s=m\,(mod\, d)
\end{array}}{n \choose s}\alpha_{m}s^{\beta_{m}}(l-1)^{s}
\]
where $\beta_{m}\in\frac{1}{2}\mathbb{Z}$ for every $m$. To estimate
this sum, let $\omega$ be a primitive $d^{th}$ root of unit and
note that $\sum_{t=0}^{d-1}\omega^{(s-m)t}$ is zero if $s\ne m\,(mod\, d)$,
and is $d$ if $s=m\,(mod\, d)$. Thus,
\[
c_{n}^{G}(W)\sim\sum_{m=0}^{d-1}\,\,\,\sum_{\begin{array}{c}
0\leq s\leq n\\
s=m\,(mod\, d)
\end{array}}{n \choose s}\alpha_{m}s^{\beta_{m}}(l-1)^{s}
\]
\[
=\frac{1}{d}\sum_{m=0}^{d-1}\sum_{s=0}^{n}\sum_{t=0}^{d-1}\omega^{(s-m)t}{n \choose s}\alpha_{m}s^{\beta_{m}}(l-1)^{s}
\]
\[
=\sum_{m=0}^{d-1}\sum_{t=0}^{d-1}\frac{\alpha_{m}}{d}\omega^{-mt}\sum_{s=0}^{n}{n \choose s}s^{\beta_{m}}\left(\omega^{t}(l-1)\right)^{s}
\]
By lemma 1.1 in \cite{Beckner&Regev}, for every $m$ and $t$ the
expression $\sum_{s=0}^{n}{n \choose s}s^{\beta_{m}}\left(\omega^{t}(l-1)\right)^{s}$
asymptotically equal to:

\[
\mu_{m,t}n^{\beta_{m}}\left(\omega^{t}(l-1)+1\right)^{n}
\]
where $\mu_{m,t}$ is a constant which does not depend on $n$. Moreover,
the absolute value of the expression $\omega^{t}(l-1)+1$ is maximal
when $t=0$, so
\[
c_{n}^{G}(W)\sim\alpha n^{\beta}l^{n}
\]
where $\beta=max(\beta_{m})\in\frac{1}{2}\mathbb{Z}$. And theorem
\nameref{thm:B} is proven. 
\selectlanguage{english}%

\end{document}